\documentclass{kms-j}

\newcommand{\publname}{}
\issueinfo{}% volume number
  {}%        % issue number
  {}%        % month
  {}%     % year
\pagespan{1}{}
\copyrightinfo{}%              % copyright year
  {Korean Mathematical Society}% copyright holder
\usepackage{graphicx}
\usepackage{amssymb}
\allowdisplaybreaks

\theoremstyle{plain}
\newtheorem{theorem}{Theorem}[section]
\newtheorem{proposition}[theorem]{Proposition}
\newtheorem{lemma}[theorem]{Lemma}

\theoremstyle{definition}

\theoremstyle{remark}

\begin{document}
\title[On weakly Einstein almost contact manifolds]
{On weakly Einstein almost contact manifolds}

\author[X.M. Chen]{Xiaomin Chen}
\address{Xiaomin Chen \\ College of Science \\ China University of Petroleum-Beijing \\ Beijing 102249, China}
\email{xmchen@cup.edu.cn}
\thanks{The author is supported by Natural Science Foundation of Beijing, China (Grant No.1194025).}


\begin{abstract}
In this article we study almost contact manifolds admitting weakly Einstein metrics.
We first prove that if a (2n+1)-dimensional Sasakian manifold admits a weakly Einstein metric then its scalar curvature $s$ satisfies $-6\leqslant s \leqslant 6$ for $n=1$ and
$-2n(2n+1)\frac{4n^2-4n+3}{4n^2-4n-1}\leqslant s \leqslant 2n(2n+1)$ for $n\geqslant2$. Secondly, for a (2n+1)-dimensional weakly Einstein contact metric $(\kappa,\mu)$-manifold with $\kappa<1$, we prove that it is flat or is locally isomorphic
to the Lie group $SU(2)$, $SL(2)$, or $E(1,1)$ for $n=1$ and that for $n\geqslant2$ there are no weakly Einstein metrics on contact metric $(\kappa,\mu)$-manifolds with $0<\kappa<1$. For $\kappa<0$, we get a classification of weakly Einstein contact metric $(\kappa,\mu)$-manifolds. Finally, it is proved that a weakly Einstein almost cosymplectic $(\kappa,\mu)$-manifold with $\kappa<0$  is locally isomorphic to a solvable non-nilpotent Lie group.
\end{abstract}

\maketitle

\section{Introduction}
An $n$-dimensional Riemannian manifold $(M,g)$ is said to be \emph{weakly Einstein} if its Riemannian tensor $R$ satisfies
\begin{equation}\label{1}
  \breve{R}=\frac{|R|^2}{n}g.
\end{equation}
Here $\breve{R}$ is a $(0,2)$-type tensor defined as
\begin{equation*}
  \breve{R}(X,Y)=\sum_{i,j,k=1}^nR(X,e_i,e_j,e_k)R(Y,e_i,e_j,e_k)
\end{equation*}
for an orthonormal frame $\{e_i\},i=1,2,\cdots,n$. The concept was introduced by Euh, Park and Sekigawa in \cite{EPS2}.
We also notice that if a weakly Einstein metric is critical to the functional
\begin{equation*}
g\mapsto \int_M|s_g|^2dv_g,
\end{equation*}
where $s_g$ is the scalar curvature of $M$ (see \cite{BE}), then it becomes an Einstein metric. Moreover, it is easy to verify that for a 4-dimensional manifold,  Einstein metrics are in fact weakly Einstein metrics.
However, when ${\rm dim}M>4$ a generic Einstein metric is not necessary a weakly Einstein metric. Based on the fact,  Hwang-Yun  considered whether an $n$-dimensional weakly
Einstein metric that is a nontrivial solution to the critical point equation is Einstein (cf.\cite{HY}). More recently, Baltazar-Silva-Oliveira \cite{BSO} classified a four dimensional weakly Einstein manifold with Miao-Tam critical metric under some assumptions on scalar curvature.

In the present paper, we study odd-dimensional manifolds with weakly Einstein metrics. First we consider a Sasakian manifold admitting a weakly Einstein metric and obtain the following result.
\begin{theorem}\label{T1.2}
Let $M^{2n+1}$ be a weakly Einstein Sasakian manifold.  Then the scalar curvature $s$ satisfies
 \begin{equation*}
\left\{
  \begin{array}{ll}
     -6\leqslant s \leqslant 6, &\hbox{for}\quad n=1;\\
     -2n(2n+1)\frac{4n^2-4n+3}{4n^2-4n-1}\leqslant s \leqslant 2n(2n+1),&\hbox{for}\quad n\geqslant2,
  \end{array}
\right.
\end{equation*}
and the right equality holds if and only if $M$ is a conformal flat Einstein manifold.
\end{theorem}

On the other hand, we observe that a remarkable class of contact metric manifolds is $(\kappa,\mu)$-space, originally introduced by
D. E. Blair, T. Koufogiorgos and V. J. Papantoniou in \cite{BKP}, whose curvature tensor
 satisfies
\begin{equation}\label{2}
  R(X,Y)\xi=\kappa(\eta(Y)X-\eta(X)Y)+\mu(\eta(Y)hX-\eta(X)hY)
\end{equation}
for any vector fields $X,Y$, where $\kappa$ and $\mu$ are constants and $h:=\frac{1}{2}\mathcal{L}_\xi\phi$ is a self-dual operator.
Moreover, Blair et al. proved the following classification theorem.
\begin{theorem}\emph{(\cite[Theorem 3]{BKP})}\label{T1.5}
 Let $M$ be a $3$-dimensional $(\kappa,\mu)$-manifold. Then $\kappa\leqslant1$. If $\kappa=1$ then $h=0$
 and $M$ is a Sasakian manifold. If $\kappa<1$ then $M$ is locally isometric to one of the unimodular Lie groups
 $SU(2)$, $SL(2,\mathbb{R}), E(2),E(1, 1)$ with a left-invariant metric.

Moreover, this structure can occur on $SU(2)$ or $SO(3)$ when $1-\lambda-\mu/2>0$ and
$1+\lambda-\mu/2>0$, on $SL(2, \mathbb{R})$ or $O(1,2)$ when $1-\lambda-\mu/2<0$ and $1+\lambda-\mu/2>0$
or $1-\lambda-\mu/2<0$ and $1+\lambda-\mu/2<0$, on $E(2)$ when $1-\lambda-\mu/2=0$ and $\mu<2$,
including a flat structure when $\mu=0$, and on $E(1, 1)$ when $1+\lambda-\mu/2=0$ and
$\mu>2$.
\end{theorem}
For a non-Sasakian $(\kappa,\mu)$-manifold $M$, Boeckx \cite{B} introduced an
invariant
\begin{equation*}
  I_M=\frac{1-\frac{\mu}{2}}{\sqrt{1-\kappa}}
\end{equation*}
and proved the following conclusion:
\begin{theorem}\emph{(\cite[Corollary 5]{B})}\label{T1.3}
Let $M$ be a non-Sasakian $(\kappa, \mu)$-space. Then it is
locally isometric, up to a $D$-homothetic transformation, to the unit tangent sphere
bundle of some space of constant curvature (different from 1) if and only if $I_M>-1$.
\end{theorem}
In view of Theorem \ref{T1.5} and Theorem \ref{T1.3}, we obtain
\begin{theorem} \label{T3}
 A $3$-dimensional weakly Einstein contact metric $(\kappa,\mu)$-manifold for $\kappa<1$ is flat, or is locally isomorphic
to the Lie group $SU(2), SL(2,\mathbb{R}), E(1, 1)$ endowed with a left-invariant metric.

For the dimensions $\geqslant5$ there are no weakly Einstein metrics on contact metric $(\kappa,\mu)$-manifolds with $0<\kappa<1$.
If $M^{2n+1}(n>1)$ is a weakly Einstein contact metric $(\kappa,\mu)$-manifold with $\kappa<0$, then $\mu>\frac{n-2+\sqrt{9n^2-16n+8}}{2n-1}$ or $\mu<\frac{n-2-\sqrt{9n^2-16n+8}}{2n-1}$. In particular, when
$\mu<\frac{n-2-\sqrt{9n^2-16n+8}}{2n-1}$, $M$ is locally isometric, up to a $D$-homothetic transformation, to the unit tangent sphere
bundle of some space of constant curvature.
\end{theorem}

Finally, we notice that Endo considered another class of odd-dimensional manifolds,
which are said to be \emph{almost cosymplectic $(\kappa,\mu)$-manifolds},
and it is proved that $\kappa\leqslant0$ and the equality holds if and only
if the almost cosymplectic $(\kappa,\mu)$-manifolds are cosymplectic (cf.\cite{E2}).
 Since Blair \cite{B} proved that a cosymplectic manifold is locally the
product of a K\"{a}hler manifold and an interval or unit circle $S^1$,
we are only require to consider the case where $\kappa<0$.
For an almost cosymplectic $(\kappa,\mu)$-manifold with $\kappa<0$,
if it is equipped with a weakly Einstein metric, we obtain the following conclusion.
\begin{theorem} \label{T2}
 A weakly Einstein almost cosymplectic $(\kappa,\mu)$-manifold for $\kappa<0$ is locally isomorphic
to a solvable non-nilpotent Lie group $G_\lambda$ endowed with an almost cosymplectic structure, where
$\lambda=\sqrt{-\kappa}$.
\end{theorem}

In order to prove these conclusions, in Section 2 we recall some basic concepts and formulas.
The proofs of theorems will be given in Section 3,
Section 4 and Section 5, respectively.

\section{Preliminaries}
\subsection{Weakly Einstein metrics}
In a local coordinate system the components of the $(0,4)$-Riemannian curvature tensor are
given by  $R_{ijkl}=g(R(e_i,e_j)e_k,e_l)$.
Throughout the paper the Einstein convention of summing over the repeated indices will
be adopted. The Ricci tensor $Ric$ is obtained by the contraction $(Ric)_{jk}=R_{jk}=g^{il}R_{ijkl}$.
$s=g^{ik}R_{ik}$ will denote the scalar curvature and $(\mathring{Ric})_{ik}=R_{ik}-\frac{s}{n}g_{ik}$ the traceless Ricci tensor.

We say that a Riemannian manifold $(M^n,g)$ is \emph{weakly Einstein} if the Riemannian tensor $R$ satisfies \eqref{1}, i.e.
\begin{equation*}
  \breve{R}_{ij}=\frac{|R|^2}{n}g_{ij}
\end{equation*}
for an orthonormal frame $\{e_i\}, i=1,2,\cdots,n$, where the 2-tensor $\breve{R}_{ij}$ is defined as $\breve{R}_{ij}=R_{ipqr}R_{jpqr}$ and $|R|^2=R_{ijkl}R_{ijkl}$.

On an $n$-dimensional Riemannian manifold $(M^n,g)$ for $n\geqslant3$, the Weyl tensor is defined by
\begin{equation}\label{3}
\begin{aligned}
W_{ijkl}=&R_{ijkl}+\frac{1}{n-2}(g_{ik}R_{jl}-g_{il}R_{jk}-g_{jk}R_{il}+g_{jl}R_{ik})\\
&-\frac{s}{(n-1)(n-2)}(g_{jl}g_{ik}-g_{il}g_{jk}).
\end{aligned}
\end{equation}
Here, we remark that the curvature tensor of Blair \cite{BL} is different from ours by a sign.
It is well-known that the Weyl tensor $W$ identically vanishes for $n=3$.
From \eqref{3},  we conclude (see \cite[Eq.(6)]{HY})
\begin{equation}\label{5}
|R|^2=\frac{2s^2}{n(n-1)}+\frac{4}{n-2}|\mathring{Ric}|^2+|W|^2,
\end{equation}
where $s$ denotes the scalar curvature of $M$ and $\mathring{Ric}=Ric-\frac{s}{n}g$ is the traceless Ricci tensor.

\subsection{Almost contact manifolds}
In the following we suppose that $M$ is a $(2n+1)$-dimensional smooth manifold.
An \emph{almost contact structure} on $M$ is a triple $(\phi,\xi,\eta)$, where $\phi$ is a
$(1,1)$-tensor field, $\xi$ a unit vector field, called Reeb vector field, $\eta$ a one-form dual to $\xi$ satisfying
$\phi^2=-I+\eta\otimes\xi,\,\eta\circ \phi=0,\,\phi\circ\xi=0.$
A smooth manifold with such a structure is called an \emph{almost contact manifold}.

A Riemannian metric $g$ on $M$ is called compatible with the almost contact structure if
\begin{equation*}
g(\phi X,\phi Y)=g(X,Y)-\eta(X)\eta(Y),\quad g(X,\xi)=\eta(X)
\end{equation*}
for any $X,Y\in\mathfrak{X}(M)$. An almost contact structure together with a compatible metric
is called an \emph{almost contact metric structure} and $(M,\phi,\xi,\eta,g)$ is called an almost contact metric manifold. Such an almost contact metric manifold
  is called \emph{contact metric manifold} if $d\eta=\omega$, where $\omega$ denotes the fundamental 2-form on $M$ defined by $\omega(X,Y):=g(\phi X,Y)$ for all $X,Y\in\mathfrak{X}(M)$. An almost contact structure $(\phi,\xi,\eta)$ is said
to be \emph{normal} if the corresponding complex structure $J$ on $M\times\mathbb{R}$ is integrable. If a contact metric manifold $M$ is normal, it is said be a \emph{Sasakian manifold}.
For a Sasakian manifold, the following equations hold (\cite{BL}):
\begin{align}
R(X,Y)\xi=&\eta(Y)X-\eta(X)Y,\label{6}\\
Q\xi=&2n\xi.\label{7}
\end{align}

Now we recall that there is an operator
$h=\frac{1}{2}\mathcal{L}_\xi\phi$ which is a self-dual operator.   For a contact metric manifold, it is proved the following relations \cite{BKP}:
\begin{equation*}
\mathrm{trace}(h)=0,\quad h\xi=0,\quad\phi h=-h\phi,\quad g(hX,Y)=g(X,hY),\quad\forall X,Y\in\mathfrak{X}(M).
\end{equation*}
We notice that the above formulas also hold in almost cosymplectic manifolds (see \cite{E1}).

As the generalization of the condition $R(X,Y)\xi=0$, Blair et al. defined a so-called  \emph{$(\kappa,\mu)$-nullity distribution} on a contact metric manifold:
\begin{align*}
N_x(\kappa,\mu):=&\{Z\in T_xM |R(X,Y )Z = \kappa(g(Y,Z)X-g(X,Z)Y)\\
&+\mu(g(Y,Z)hX-g(X,Z)hY)\}
\end{align*}
for two real numbers $\kappa,\mu\in\mathbb{R}$ (see \cite{BKP}). A contact metric manifold is called a \emph{contact metric $(\kappa,\mu)$-manifold} if the Reeb vector field $\xi$ belongs to the $(\kappa,\mu)$-nullity distribution, namely the condition \eqref{2} is satisfied.

Let $(M,\phi,\xi,\eta)$ be an almost contact metric manifold. If the fundamental 2-form $\omega$ and 1-form $\eta$ are closed, then $M$ is called an \emph{almost cosymplectic manifold}. Moreover, if $M$ is normal, it is said to be cosymplectic. An \emph{almost cosymplectic $(\kappa,\mu)$-manifold} is
an almost cosymplectic manifold satisfying \eqref{2}. This class of almost contact manifold was firstly considered in \cite{E2}.  In particular, any cosymplectic manifold is an almost cosymplectic $(\kappa,\mu)$-manifod with $\kappa=0$ and any $\mu$. Endo proved that if $\kappa\neq0$ any almost cosymplectic $(\kappa,\mu)$-manifolds are not cosymplectic (\cite{E2}). When $\kappa<0$ and $\mu=0$, Dacko in \cite{D} proved that $M$ is necessarily an almost cosymplectic manifold with K\"{a}hlerian leaves, moreover gave a full description of the local structure of this class.

We briefly recall the structure, referring to \cite{D} for more details. Let $\lambda$ be a real positive number and $\mathfrak{g}_\lambda$ be
the solvable non-nilpotent Lie algebra with basis $\{\xi,X_1,\cdots,X_n, Y_1,\cdots,Y_n\}$ and
non-zero Lie brackets
\begin{equation*}
[\xi,X_i]=-\lambda X_i,\quad [\xi,Y_i]=\lambda Y_i,
\end{equation*}
for each $i\in\{1,\cdots,n\}$. Let $G_\lambda$ be a Lie group whose Lie algebra is $\mathfrak{g}_\lambda$ and let
$(\phi,\xi,\eta,g)$ be the left-invariant almost cosymplectic structure defined by
\begin{align*}
g(X_i,X_j)=&g(Y_i,Y_j)=\delta_{ij},\quad g(X_i,Y_j)=0,\quad g(\xi,X_i)=g(\xi,Y_i)=0,\\
&\phi\xi=0,\quad \phi X_i=Y_i,\quad \phi Y_i=-X_i, \eta=g(\cdot,\xi).
\end{align*}
\begin{theorem}\emph{(\cite[Theorem 4]{D})}\label{T2.1}
 An almost cosymplectic $(\kappa,0)$-manifold for some $\kappa<0$ is locally isomorphic
to the above Lie group $G_\lambda$ endowed with the above almost cosymplectic structure, where
$\lambda=\sqrt{-\kappa}$.
\end{theorem}

Throughout this paper we write the indices $i,j,k,l\in\{0,1,2,\cdots,2n\}$, $a,b,c,d\in\{1,2,\cdots,2n\}$,
$\alpha,\beta,\gamma,\delta\in\{1,2,\cdots,n\}$ and $A,B,C,D\in\{n+1,n+2,\cdots,2n\}$.
Write
\begin{equation*}
h_{ij}=g(he_i,e_j),\quad R_{ij0k}=g(R(e_i,e_j)\xi,e_k),\quad Ric(\xi,\xi)=R_{00}=R_{i00i}.
\end{equation*}

\section{Proof of  Theorem \ref{T1.2}}
In this section we assume that $M$ is a $2n+1$-dimensional Sasakian manifold and
$\{e_i\}_{i=0}^{2n}$ is a local orhonormal frame of $M$ such that $e_0=\xi,e_{n+i}=\phi e_i$ for $i=1,2,\cdots,n$.

Using \eqref{6} and \eqref{7}, it follows from \eqref{3} that
\begin{equation}\label{8}
\begin{aligned}
W_{ij0l}=&R_{ij0l}+\frac{1}{2n-1}(g_{i0}R_{jl}-g_{il}R_{j0}-g_{j0}R_{il}+g_{jl}R_{i0})\\
&-\frac{s}{2n(2n-1)}(g_{jl}g_{i0}-g_{il}g_{j0})\\
=&\Big[1-\frac{1}{2n-1}\Big(2n-\frac{s}{2n}\Big)\Big](g_{j0}g_{il}-g_{i0}g_{jl})+\frac{1}{2n-1}(g_{i0}R_{jl}-g_{j0}R_{il}).
\end{aligned}
\end{equation}
Since $W$ is totally trace-free, we have
\begin{align}\label{9}
  W_{ija0}W_{ija0} =& \Big[1-\frac{1}{2n-1}\Big(2n-\frac{s}{2n}\Big)\Big](g_{j0}g_{ia}-g_{i0}g_{ja})W_{ij0a}\nonumber\\
  &+\frac{1}{2n-1}(g_{i0}R_{ja}-g_{j0}R_{ia})W_{ij0a} \nonumber\\
  = & \Big[1-\frac{1}{2n-1}\Big(2n-\frac{s}{2n}\Big)\Big](W_{b00a}g_{ba}-W_{0b0a}g_{ba})\\
  &+\frac{1}{2n-1}(W_{0j0a}R_{ja}-W_{i00a}R_{ia})\nonumber\\
   = & 2\Big[1-\frac{1}{2n-1}\Big(2n-\frac{s}{2n}\Big)\Big]W_{a00a}+\frac{2}{2n-1}W_{0b0a}R_{ba}\nonumber\\
   =&-\frac{2}{2n-1}\Big[1-\frac{1}{2n-1}\Big(2n-\frac{s}{2n}\Big)\Big]R_{aa}+\frac{2}{(2n-1)^2}R^2_{ba}\nonumber\\
    =&-\Big[1-\frac{1}{2n-1}\Big(2n-\frac{s}{2n}\Big)\Big]\frac{2(s-2n)}{2n-1}\nonumber\\
   &+\frac{2}{(2n-1)^2}\Big(-4n^2+|\mathring{Ric}|^2+\frac{s^2}{2n+1}\Big)\nonumber\\
   =&\frac{2}{(2n-1)^2}\Big[|\mathring{Ric}|^2-\frac{s^2}{2n(2n+1)}+2s-2n(2n+1)\Big].\nonumber
\end{align}
Here we have used $R_{aa}=s-2n$ and $\mathring{Ric}=Ric-\frac{s}{2n+1}$.

On the other hand, using \eqref{6} we directly compute
\begin{align*}
  \breve{R}(\xi,\xi)= &R_{0ijk}R_{0ijk}=(g_{0j}g_{ik}-g_{0i}g_{jk})R_{ij0k} \\
  = & R_{i00i}-R_{0j0j}=2R_{00}=4n.
\end{align*}
Therefore for a $2n+1$-dimensional contact metric manifold, \eqref{5} should become
\begin{equation}\label{11}
  4n=\frac{1}{2n+1}\Big(\frac{2s^2}{2n(2n+1)}+\frac{4}{2n-1}|\mathring{Ric}|^2+|W|^2\Big).
\end{equation}
It is well know that when $n=1$, $W=0$, then
\begin{equation*}
 4|\mathring{Ric}|^2=12-\frac{s^2}{3}.
\end{equation*}
Thus we have
\begin{equation*}
-6\leqslant s\leqslant 6.
\end{equation*}

Next we assume $n\geqslant2$ and the following lemma is clear.
\begin{lemma}\label{L3.1}
\begin{equation*}
|W|^2=2W_{ija0}W_{ija0}+W_{dcab}W_{dcab}.
\end{equation*}
\end{lemma}
\begin{proof}For the indices $i,j,k,l\in\{0,1,2,\cdots,2n\}$ and $a,b,c,d\in\{1,2,\cdots,2n\}$, we directly compute
\begin{align*}
  |W|^2= &W_{ijkl}W_{ijkl}= W_{ij0l}W_{ij0l}+W_{ijal}W_{ijal} \\
   =& W_{ij0l}W_{ij0l}+W_{ija0}W_{ija0}+W_{ijab}W_{ijab}\\
    =& 2W_{ija0}W_{ija0}+W_{ijab}W_{ijab}\\
   =& 2W_{ija0}W_{ija0}+W_{d0ab}W_{d0ab}+W_{0cab}W_{0cab}+W_{dcab}W_{dcab}\\
   =& 2W_{ija0}W_{ija0}+2W_{d0ab}W_{d0ab}+W_{dcab}W_{dcab}\\
   =& 2W_{ija0}W_{ija0}+W_{dcab}W_{dcab}
\end{align*}
since $W_{d0ab}=0$ for any $a,b,d\in\{1,2,\cdots,2n\}$ by \eqref{8} .
\end{proof}

By Lemma \ref{L3.1}, substituting \eqref{9} into \eqref{11} gives
\begin{align*}
  4n(2n+1)= & \frac{2s^2}{2n(2n+1)}+\frac{4}{2n-1}|\mathring{Ric}|^2+\frac{4}{(2n-1)^2}\Big[|\mathring{Ric}|^2\\
  &-\frac{s^2}{2n(2n+1)}+2s-2n(2n+1)\Big]+W_{dcab}W_{dcab}.
\end{align*}
Since $W_{dcab}W_{dcab}\geqslant0$, we conclude
\begin{align*}
  0 \geqslant  \frac{4n^2-4n-1}{n(2n+1)}s^2+8n|\mathring{Ric}|^2+8s-4n(2n+1)(4n^2-4n+3),
\end{align*}
that is,
\begin{align*}
  4n|\mathring{Ric}|^2\leqslant-\frac{4n^2-4n-1}{2n(2n+1)}s^2-4s+2n(2n+1)(4n^2-4n+3).
\end{align*}
Hence
\begin{equation*}
-\frac{4n^2-4n-1}{2n(2n+1)}s^2-4s+2n(2n+1)(4n^2-4n+3)\geqslant0.
\end{equation*}
 Because $n\geqslant2$, we obtain
\begin{align*}
-2n(2n+1)\frac{4n^2-4n+3}{4n^2-4n-1}\leqslant s \leqslant 2n(2n+1).
\end{align*}
Moreover, when the right equality holds, from \eqref{11} we find $W=0$, i.e. $M$ is conformal flat. We complete the proof of Theorem \ref{T1.2}.
\section{Proof of Theorem \ref{T3}}
In this section we suppose that $(M^{2n+1},\phi,\eta,\xi,g)$ is a contact metric $(\kappa,\mu)$-manifold.
It is proved that  $\kappa\leqslant1$ and if $\kappa=1$ then $h=0$, i.e. $M$ is a Sasakian manifold by Theorem \ref{T1.5}. Thus we only need to consider the case where $\kappa<1.$ For the contact metric $(\kappa,\mu)$-manifold the following
 lemma was given.
 \begin{lemma}[\cite{KMP}]\label{L2.2}
Let $(M^{2n+1},\phi,\eta,\xi,g)$ be a contact metric $(\kappa,\mu)$-manifold with $\kappa<1$. For every $p\in M$, there exists an open neighborhood $W$ of $p$ and
orthonormal local vector fields $X_i,\phi X_i$, and $\xi$ for $i=1,\cdots,n$, defined on $W$,
such that
\begin{equation*}
 hX_i=\lambda X_i,\quad h\phi X_i=-\lambda\phi X_i,\quad h\xi = 0,
\end{equation*}
for $i=1,\cdots,n$, where $\lambda=\sqrt{1-\kappa}.$
\end{lemma}

By Lemma \ref{L2.2}, we can take a local orthonormal frame  $\{e_0=\xi,e_1,\cdots,e_{2n}\}$ of $M$ such that $e_{n+i}=\phi e_i$  and
$he_i=\lambda e_i$ and $he_{n+i}=-\lambda e_{n+i}$ for $i=1,2,\cdots,n$.

If $M$ admits a weakly Einstein metric, by \eqref{1} we have
\begin{equation}\label{12}
  \breve{R}(\xi,\xi)=\frac{1}{2n+1}|R|^2.
\end{equation}
We first compute $\breve{R}(\xi,\xi)$.
Since $h^2=(\kappa-1)\phi^2$ (see \cite[Eq.(3.3)]{KT}), by \eqref{2} we obtain
\begin{equation}\label{13}
\begin{aligned}
  \breve{R}(\xi,\xi)= & R_{ij0k}R_{ij0k}=[\kappa(g_{0j}g_{ik}-g_{0i}g_{jk})+\mu(g_{0j}h_{ik}-g_{0i}h_{jk})]R_{ij0k} \\
  =&2\kappa R_{i00i}+2\mu(R_{i00k}h_{ik})\\
  =&2\kappa R_{00}+2\mu[\kappa(g_{ik}-g_{0i}g_{0k})+\mu(h_{ik})]h_{ik}\\
  = & 4n(\kappa^2-\mu^2(\kappa-1)).
\end{aligned}
\end{equation}

We can prove the following lemma.
\begin{lemma}\label{L4.2}
\begin{equation*}
\begin{aligned}
 |R|^2=&2\breve{R}(\xi,\xi)+ R_{\alpha\beta\delta\gamma}R_{\alpha\beta\delta\gamma}+4R_{\alpha\beta\delta A}R_{\alpha\beta\delta A}+2R_{\alpha\beta AB}R_{\alpha\beta AB}\\
    &+4R_{\alpha A\beta B}R_{\alpha A\beta B}+4R_{AB\alpha C}R_{AB\alpha C}+R_{ABCD}R_{ABCD}.
\end{aligned}
\end{equation*}
\end{lemma}
\begin{proof}
First, similar to the proof of Lemma \ref{L3.1} we derive
\begin{equation}\label{15}
  |R|^2=R_{ijkl}R_{ijkl}=2R_{ija0}R_{ija0}+R_{abcd}R_{abcd}=2\breve{R}(\xi,\xi)+R_{abcd}R_{abcd}.
\end{equation}
Moreover, we compute
\begin{align*}
  R_{cdab}R_{cdab}= &R_{\alpha dab}R_{\alpha dab}+R_{Adab}R_{Adab}  \\
  = & R_{\alpha\beta ab}R_{\alpha \beta ab}+R_{\alpha A ab}R_{\alpha A ab}+R_{A\alpha ab}R_{A\alpha ab}+R_{ABab}R_{ABab}\\
   = & R_{\alpha\beta \delta b}R_{\alpha \beta \delta b}+R_{\alpha\beta Ab}R_{\alpha \beta Ab}+2(R_{\alpha A \beta b}R_{\alpha A \beta b}+R_{\alpha A Bb}R_{\alpha A Bb})\\
   &+R_{AB\alpha b}R_{AB\alpha b}+R_{ABCb}R_{ABCb}\\
    = & R_{\alpha\beta\delta\gamma}R_{\alpha\beta\delta\gamma}+R_{\alpha\beta\delta A}R_{\alpha\beta\delta A}+R_{\alpha\beta A\delta}R_{\alpha \beta A\delta}+R_{\alpha\beta AB}R_{\alpha \beta AB}\\
    &+2(R_{\alpha A \beta \delta}R_{\alpha A \beta \delta}+R_{\alpha A \beta B}R_{\alpha A \beta B}+R_{\alpha A B\beta}R_{\alpha A B\beta}+R_{\alpha A BC}R_{\alpha A BC})\\
   &+R_{AB\alpha \beta}R_{AB\alpha \beta}+R_{AB\alpha C}R_{AB\alpha C}+R_{ABC\alpha}R_{ABC\alpha}+R_{ABCD}R_{ABCD}\\
    = & R_{\alpha\beta\delta\gamma}R_{\alpha\beta\delta\gamma}+4R_{\alpha\beta\delta A}R_{\alpha\beta\delta A}+2R_{\alpha\beta AB}R_{\alpha\beta AB}\\
    &+4R_{\alpha A\beta B}R_{\alpha A\beta B}+4R_{AB\alpha C}R_{AB\alpha C}+R_{ABCD}R_{ABCD}.
\end{align*}
We complete the proof the lemma by substituting the above formula into \eqref{15}.
\end{proof}

\begin{proposition}\emph{(\cite[Theorem 1]{BKP})}\label{P2}
 Let $M^{2n+1}(\phi,\eta,\xi,g)$ be a contact metric manifold with
belonging to the $(\kappa, \mu)$-nullity distribution. If $\kappa<1$, $M^{2n+1}$ admits three mutually
orthogonal and integrable distributions $\mathcal{D}(0), \mathcal{D}(\lambda)$ and $\mathcal{D}(-\lambda)$ determined by
the eigenspaces of $h$, where $\lambda=\sqrt{1-\kappa}$. Moreover,
\begin{align*}
R(X_\lambda, Y_\lambda)Z_{-\lambda} &= (\kappa-\mu)[g(\phi Y_\lambda, Z_{-\lambda})\phi X_\lambda-g(\phi X_\lambda, Z_{-\lambda})\phi Y_\lambda],\\
R(X_{-\lambda}, Y_{-\lambda})Z_\lambda &= (\kappa-\mu)[g(\phi Y_{-\lambda}, Z_\lambda)\phi X_{-\lambda}-g(\phi X_{-\lambda}, Z_\lambda)\phi Y_{-\lambda}],\\
R(X_\lambda, Y_{-\lambda})Z_{-\lambda} &=\kappa g(\phi X_\lambda, Z_{-\lambda})\phi Y_{-\lambda}+\mu g(\phi X_\lambda, Y_{-\lambda})\phi Z_{-\lambda},\\
R(X_\lambda, Y_{-\lambda})Z_{\lambda}&=-\kappa g(\phi Y_{-\lambda}, Z_\lambda)\phi X_\lambda-\mu g(\phi Y_{-\lambda}, X_{\lambda})\phi Z_\lambda,\\
R(X_\lambda, Y_{\lambda})Z_{\lambda} &= [2(1+\lambda)- \mu][g(Y_{\lambda}, Z_{\lambda})X_{\lambda}- g(X_{\lambda}, Z_{\lambda})Y_{\lambda}],\\
R(X_{-\lambda}, Y_{-\lambda})Z_{-\lambda} &= [2(1-\lambda)- \mu][g(Y_{-\lambda}, Z_{-\lambda})X_{-\lambda}-g(X_{-\lambda}, Z_{-\lambda})Y_{-\lambda}],
\end{align*}
where $X_\lambda, Y_\lambda, Z_\lambda\in \mathcal{D}(\lambda)$ and $X_{-\lambda}, Y_{-\lambda}, Z_{-\lambda}\in \mathcal{D}(-\lambda)$.
\end{proposition}
By Proposition \ref{P2}, we can get
\begin{align*}
  R_{\alpha\beta\delta\gamma} & =[2(1+\lambda)- \mu](g_{\beta\delta}g_{\alpha\gamma}- g_{\alpha\delta}g_{\beta\gamma}),\\
  R_{\alpha\beta\delta A} & =0,\\
  R_{\alpha\beta AB}&=(\kappa-\mu)(g_{\overline{\beta}A}g_{\overline{\alpha}B}-g_{\overline{\alpha}A}g_{\overline{\beta}B}),\\
  R_{\alpha A\beta B}&=-\kappa g_{\overline{A}\beta}g_{\overline{\alpha}B}-\mu g_{\overline{A}\alpha}g_{\overline{\beta} B},\\
  R_{AB\alpha C}&=0,\\
  R_{ABCD}&=[2(1-\lambda)- \mu](g_{BC}g_{AD}- g_{AC}g_{BD}).
\end{align*}
Hence
\begin{align*}
  R_{\alpha\beta\delta\gamma}^2 & =2n(n-1)[2(1+\lambda)- \mu]^2, \\
  R_{\alpha\beta\delta A}^2 & =0,\\
  R_{\alpha\beta AB}^2&=2n(n-1)(\kappa-\mu)^2,\\
 R_{\alpha A\beta B}^2&=(\kappa^2+\mu^2)n^2+2n\kappa\mu,\\
  R_{AB\alpha C}^2&=0,\\
  R_{ABCD}^2&=2n(n-1)[2(1-\lambda)- \mu]^2.
\end{align*}
Therefore by Lemma \ref{L4.2} and \eqref{13} we conclude
\begin{equation}\label{16}
\begin{aligned}
  |R|^2=&8n(\kappa^2-(k-1)\mu^2)+2n(n-1)[2(1+\lambda)- \mu]^2+4n(n-1)(\kappa-\mu)^2\\
    &+4[(\kappa^2+\mu^2)n^2+2n\kappa\mu]+2n(n-1)[2(1-\lambda)- \mu]^2.
\end{aligned}
\end{equation}
Substituting \eqref{16} into \eqref{12} and using \eqref{13}, we have
\begin{equation}\label{17}
-(2n-1)\mu^2\kappa=(n-1)[4(1+\lambda^2)+\mu^2-4\mu]-2(n-2)\kappa\mu.
\end{equation}

Now we divide into two cases to discuss.

\textbf{Case I:} $n=1$. Then \eqref{17} implies $(\mu+2)\mu\kappa=0$. If $\kappa=\mu=0$, $M$ is flat (see \cite[Theorem 7.5]{BL}).

By Theorem \ref{T1.5}, when $\kappa=0,\mu\neq0$, then $1+\lambda-\frac{\mu}{2}=2-\frac{\mu}{2}, 1-\lambda-\frac{\mu}{2}=-\frac{\mu}{2}$, and $M$ is locally isometric to the Lie group $SU(2)$, $SL(2,\mathbb{R})$ or $E(1,1)$.

When $0\neq\kappa<1$ and $\mu=0$, we know $1+\lambda-\frac{\mu}{2}=1+\lambda>0$. When $0\neq\kappa<1$ and $\mu=-2$, then $1+\lambda-\mu/2=2+\lambda>0$.
Both cases imply that $M$ is locally isometric to the Lie group $SU(2)$ or $SL(2,\mathbb{R})$ by Theorem \ref{T1.5}.

\textbf{Case II:} $n>1$. Since $\lambda^2=1-k$, it follows from \eqref{17} that
\begin{equation}\label{18}
[-\mu^2(2n-1)+2\mu(n-2)+4(n-1)]\kappa=(n-1)[4+(\mu-2)^2].
\end{equation}

Moreover, when $0<\kappa<1$, we find
\begin{align*}
[-\mu^2(2n-1)+2\mu(n-2)+4(n-1)]>&(n-1)[4+(\mu-2)^2].
\end{align*}
That is,
\begin{align*}
(3n-2)\mu^2-2(3n-4)\mu+4(n-1)<0.
\end{align*}
Because $n>1$, it is easy to prove that the above inequality has no solution.

When $\kappa<0$, Equation \eqref{18} implies
\begin{equation*}
-\mu^2(2n-1)+2\mu(n-2)+4(n-1)<0,
\end{equation*}
that is,
\begin{align*}
\mu>\frac{n-2+\sqrt{9n^2-16n+8}}{2n-1}\quad\hbox{or}\quad\mu<\frac{n-2-\sqrt{9n^2-16n+8}}{2n-1}.
\end{align*}
In particular, when $\mu<\frac{n-2-\sqrt{9n^2-16n+8}}{2n-1}$, we know $\mu<0$ since $n>1$.
Hence the invariant $I_M$ (see introduction) must be greater than $-1$.
Therefore, we complete the proof by Theorem \ref{T1.3}.

\section{Proof of Theorem \ref{T2}}
In this section let us assume that $(M^{2n+1},\phi,\eta,\xi,g)$ is an almost cosymplectic $(\kappa,\mu)$-manifold, namely an almost cosymplectic manifold
satisfies \eqref{2}. First the following relation are provided (see \cite[Eq.(3.23)]{MNY}):
\begin{equation}\label{19}
  h^2=\kappa\phi^2,
\end{equation}
where $Q$ is the Ricci operator defined by $\mathrm{Ric}(X,Y)=g(QX,Y)$ for any vectors $X,Y$. In particular,  $Q\xi=2n\kappa\xi$ because of $h\xi=0$.
From \eqref{19}, $trace(h^2)=-2n\kappa$. Furthermore, since $\kappa\phi^2=h^2$, $\kappa\leq0$ and the equality holds if and only if the almost cosymplectic $(\kappa,\mu)$-manifolds are cosymplectic. Therefore, we will concentrate
on the case $\kappa<0$.

Since $trace(h^2)=-2n\kappa$, as the calculation of \eqref{13}, making use of \eqref{2} we obtain
\begin{equation}\label{20}
  \breve{R}(\xi,\xi)= 4n(\kappa^2-\mu^2\kappa).
\end{equation}

For an almost cosymplectic $(\kappa,\mu)$-manifold with $\kappa<0$, we also have a similar lemma to Lemma \ref{L2.2}.
 \begin{lemma}\label{L}
Let $(M^{2n+1},\phi,\eta,\xi,g)$ be a almost cosymplectic $(\kappa,\mu)$-manifold with $\kappa<0$. For every $p\in M$, there exists an open neighborhood $W$ of $p$ and
orthonormal local vector fields $X_i,\phi X_i$, and $\xi$ for $i=1,\cdots,n$, defined on $W$,
such that
\begin{equation*}
 hX_i=\lambda X_i,\quad h\phi X_i=-\lambda\phi X_i,\quad h\xi = 0,
\end{equation*}
for $i=1,\cdots,n$, where $\lambda=\sqrt{-\kappa}.$
\end{lemma}
Thus we can also take a local frame $\{e_i\}$ of $M$ as in Section 4. In this section we will adopt the same index as Section 4.
In the following we compute the square $|R|^2$  of  curvature tensor $R$. In order to do that, we notice the following proposition.

\begin{proposition}\emph{(\cite[Theorem 3.7]{CMM})}\label{P1}
Let $M$ be an almost cosymplectic $(\kappa,\mu)$-manifold of dimension
greater than or equal to $5$ with $\kappa<0$. Then its Riemann curvature tensor
can be written as
\begin{equation*}
R=-\kappa R_3-R_{5,2}-\mu R_6,
\end{equation*}
where
\begin{align*}
  R_3(X,Y)Z&=\eta(X)\eta(Z)Y-\eta(Y)\eta(Z)X+g(X,Z)\eta(Y)\xi-g(Y,Z)\eta(X)\xi, \\
  R_6(X,Y)Z&=\eta(X)\eta(Z)hY-\eta(Y)\eta(Z)hX+g(hX,Z)\eta(Y)\xi-g(hY,Z)\eta(X)\xi,\\
  R_{5,2}(X,Y)Z&=g(\phi hY,Z)\phi hX-g(\phi hX,Z)\phi hY
\end{align*}
for any vector fields $X,Y,Z$.
\end{proposition}
In view of Lemma \ref{L}, $he_a=\lambda_ae_a$ with $\lambda_a=\pm\sqrt{-\kappa}$, thus by Proposition \ref{P1}, we know
\begin{equation}\label{21}
\begin{aligned}
R_{abcd}=&-(h_{b\overline{c}}h_{a\overline{d}}-h_{a\overline{c}}h_{b\overline{d}})=-\lambda_b\lambda_a[g_{b\overline{c}}g_{a\overline{d}}-g_{a\overline{c}}g_{b\overline{d}}],
\end{aligned}
\end{equation}
where $h_{a\overline{d}}=g(he_a,\phi e_d)$ and $g_{b\overline{c}}=g(e_b,\phi e_c)$ for all $a,b,c,d\in\{1,2,\cdots,2n\}.$

Making use of \eqref{21}, we have
\begin{align*}
  R_{\alpha\beta\delta\gamma} & =0,\quad  R_{\alpha\beta\delta A}  =0,\\
  R_{\alpha\beta AB}&=-\kappa(g_{\beta\overline{A}}g_{\alpha\overline{B}}-g_{\alpha\overline{A}}g_{\beta\overline{B}}),\\
  R_{\alpha A\beta B}&=\kappa g_{A\overline{\beta}}g_{\alpha\overline{B}},\\
  R_{AB\alpha C}&=0,\quad  R_{ABCD}=0.
\end{align*}
Hence
\begin{align*}
  R_{\alpha\beta\delta\gamma}^2 & =0, \quad  R_{\alpha\beta\delta A}^2  =0,\\
  R_{\alpha\beta AB}^2&=2n(n-1)\kappa^2, \\
  R_{\alpha A\beta B}^2&=n^2\kappa^2,\\
  R_{AB\alpha C}^2&=0,\quad  R_{ABCD}^2=0.
\end{align*}
Hence we derive from \eqref{20} and Lemma \ref{L4.2} that
\begin{align*}
 |R|^2=8n(\kappa^2-\mu^2\kappa)+2\kappa^22n(n-1)+4\kappa^2n^2=4n[(2n+1)\kappa^2-2\mu^2\kappa].
\end{align*}
By \eqref{12}, we have
\begin{align*}
  (\kappa^2-\kappa\mu^2)(2n+1)=(2n+1)\kappa^2-2\mu^2\kappa.
\end{align*}
This shows $\mu=0$ since $\kappa<0$.

We complete the proof Theorem \ref{T2} by Theorem \ref{T2.1}.
\section*{Acknowledgement}
The author would like to thank the referee for the
valuable comments on this paper.

%-------------------------------------------------------------------

%-------------------------------------------------------------------
%
% Here start the references; below, we include several formatting examples
% The entries should be arranged alphabetically, by the family name of the first author
% The titles which are issued in a foreign language should be replaced by their English translation
%    (please check the third example).
%

%
% Below are the data of a book. Please note that:
%
% - the family name goes last, and that the names are followed by comma
% - the title is first-caps, written in italics
% - then go the editors, city/cities, year. The "city/cities" field is not compulsory.
%

%-------------------------------------------------------------------
---------------------------------------------------------------

\begin{thebibliography}{99}
%
\bibitem{BE}M. Berger, {\em Quelques formules de variation pour une structure riemannienne}, Ann. Sci.
\'{E}cole Norm. Sup. \textbf{4 }(1970), No. 3, 285-294.
\bibitem{BL}D. E. Blair,  {\em Riemannian Geometry of Contact and Symplectic Manifolds}. Birkh\"{a}user, Boston, 2010.
\bibitem{BKP}D. E. Blair, T. Koufogiorgos, B. J. Papantoniou, {\em Contact metric manifolds satisfying a nullity
condition,} Israel J. Math. \textbf{91} (1995), 189-214.
\bibitem{B}E. Boeckx, {\em A full clasification of contact metric $(\kappa,\mu)$-spaces.} Illinois J. Math. \textbf{44} (2000), No. 1, 212-219 .
\bibitem{BSO}H. Baltazar, A. Da Silva, F. Oliveira, {\em Weakly Einstein critical metrics of the volume functional on compact manifolds with boundary,} arXiv:1804.10706v2.
\bibitem{CMM}A. Carriazo, V. Mart\'{\i}n-Molina, {\em Almost Cosymplectic and Almost Kenmotsu $(\kappa,\mu, \nu)$-Spaces.} Mediterr. J. Math. \textbf{10} (2013), 1551-1571.
\bibitem{D}P. Dacko, {\em On almost cosymplectic manifold with structure vector field belongs to $\kappa$-distribution}, Balk. J. Geom. Appl. \textbf{5} (2000), No. 2 , 47-60.
\bibitem{E1}H. Endo, {\em On Ricci curvatures of almost cosymplectic manifolds}, An. \c{S}tiin\c{t}.
Univ. Al. I. Cuza I\c{a}si. Mat. (N.S.) \textbf{40 }(1994), 75-83.
\bibitem{E2}H. Endo, {\em Non-existence of almost cosymplectic manifolds satisfying a certain condition,}  Tensor(N.S). \textbf{63} (2002), No. 3,  272-284.


\bibitem{EPS2} Y. Euh, J. H. Park, and K. Sekigawa, {\em A curvature identity on a 4-dimensional Riemannian manifold,} Results Math. \textbf{63} (2013), No. 1-2, 107-114.
\bibitem{KMP}T. Koufogiorgos, M. Markellos, V. J. Papantoniou, {\em The harmonicity of the Reeb vector fields on
contact metric 3-manifolds}. Pacific J. Math. \textbf{234} (2008), No. 2, 325-344 .
\bibitem{KT} T. Koufogiorgos, C. Tsichlias,  {\em On the existence of a new class of contact metric manifolds.} Can.
Math. Bull. \textbf{43} (2000), No. 4, 400-447 .
\bibitem{HY}S. Hwang, G. Yun, {\em Weakly Einstein critical point equation,} Bull. Korean Math. Soc. \textbf{53 }(2016), No. 4, 1087-1094.

 \bibitem{MNY}B. C. Montano, A. D. Nicola, A. I. Yudin, {\em A survey on cosymplectic geometry,} Rev. Math. Phys. \textbf{25} (2013), No. 10, 1343002.










\end{thebibliography}
\end{document}